\documentclass[11pt]{amsart}
\newtheorem{theorem}{Theorem}[section]

\newtheorem{lemma}[theorem]{Lemma}
\newtheorem{proposition}[theorem]{Proposition}

\theoremstyle{definition}

\numberwithin{equation}{section}

\newcommand{\real}{\mathbb R}

\begin{document}

\title[Behavior near the origin of $f'(u^\ast)$ in extremal solutions] {Behavior near the origin of $f'(u^\ast)$ in radial singular extremal solutions}
\author{Salvador Villegas}
\thanks{The author has been supported by the Ministerio de Ciencia, Innovaci\'on y Universidades of Spain PGC2018-096422-B-I00 and by the Junta de Andaluc\'{\i}a  A-FQM187-UGR18.}
\address{Departamento de An\'{a}lisis
Matem\'{a}tico, Universidad de Granada, 18071 Granada, Spain.}
\email{svillega@ugr.es}

\begin{abstract}
Consider the semilinear elliptic equation $-\Delta u=\lambda f(u)$ in the unit ball $B_1\subset \mathbb{R}^N$, with Dirichlet
data $u|_{\partial B_1}=0$, where $\lambda\geq 0$ is a real parameter and $f$ is a $C^1$ positive, nondecreasing and convex function in $[0,\infty)$ such that
$f(s)/s\rightarrow\infty$ as $s\rightarrow\infty$. In this paper we study the behavior of $f'(u^\ast)$ near the origin when $u^\ast$, the extremal solution of the previous problem associated to $\lambda=\lambda^\ast$, is singular. This answers to an open problems posed by Brezis and V\'azquez \cite[Open problem 5]{BV}.
\end{abstract}

\maketitle

\section{Introduction and main results}

Consider the following semilinear elliptic equation, which has been extensively studied:
$$
\left\{
\begin{array}{ll}
-\Delta u=\lambda f(u)\ \ \ \ \ \ \  & \mbox{ in } \Omega \, ,\\
u>0 & \mbox{ in } \Omega \, ,\\
u=0  & \mbox{ on } \partial\Omega \, ,\\
\end{array}
\right. \eqno{(P_\lambda)}
$$
\

\noindent where $\Omega\subset\real^N$ is a smooth bounded domain,
$N\geq 1$, $\lambda\geq 0$ is a real parameter and the
nonlinearity $f:[0,\infty)\rightarrow \real$ satisfies

\begin{equation}\label{convexa}
f \mbox{ is } C^1, \mbox{ nondecreasing and convex, }f(0)>0,\mbox{
and }\lim_{u\to +\infty}\frac{f(t)}{t}=+\infty.
\end{equation}

It is well known that there exists a finite positive extremal
parameter $\lambda^\ast$ such that ($P_\lambda$) has a minimal
classical solution $u_\lambda\in C^0(\overline{\Omega})\cap C^2(\Omega)$ if $0<
\lambda <\lambda^\ast$, while no solution exists, even in the weak
sense, for $\lambda>\lambda^\ast$. The set $\{u_\lambda:\, 0<
\lambda < \lambda^\ast\}$ forms a branch of classical solutions
increasing in $\lambda$. Its increasing pointwise limit
$u^\ast(x):=\lim_{\lambda\uparrow\lambda^\ast}u_\lambda(x)$ is a
weak solution of ($P_\lambda$) for $\lambda=\lambda^\ast$, which
is called the extremal solution of ($P_\lambda$) (see
\cite{Bre,BV,Dup}). 

The regularity and properties of extremal solutions depend
strongly on the dimension $N$, domain $\Omega$ and nonlinearity
$f$. When $f(u)=e^u$, it was proven that $u^\ast\in L^\infty
(\Omega)$ if $N<10$ (for every $\Omega$) (see \cite{CrR,MP}),
while $u^\ast (x)=-2\log \vert x\vert$ and $\lambda^\ast=2(N-2)$
if $N\geq 10$ and $\Omega=B_1$ (see \cite{JL}). There is an
analogous result for $f(u)=(1+u)^p$ with $p>1$ (see \cite{BV}).
Brezis and V\'azquez \cite{BV} raised the question of determining
the boundedness of $u^\ast$, depending only on the dimension $N$, for general smooth bounded domains $\Omega\subset\real^N$ and nonlinearities $f$ satisfying (\ref{convexa}). This was proven by Nedev \cite{Ne} when  $N\leq3$; by Cabr\'e and Capella \cite{cc} when $\Omega=B_1$ and $N\leq 9$; by Cabr\'e \cite{ca4}  when $N=4$ and $\Omega$ is convex; by the author \cite{yo4} when $N=4$; by Cabr\'e and Ros-Oton \cite{caro} when $N\leq 7$ and $\Omega$ is a convex domain “of double revolution”; by Cabr\'e, Sanch\'on, and Spruck \cite{css} when $N=5$ and $\limsup_{t\to\infty}f'(t)/f(t)^{1+\varepsilon}<+\infty$ for every $\varepsilon>0$. Finally, in a recent paper Cabr\'e, Figalli, Ros-Oton and Serra \cite{cfros} solved completely this question by proving that $u^\ast$ is bounded if $N\leq 9$.

Another question posed by Brezis and V\'azquez \cite[Open problem 5]{BV} for singular extremal solutions is the following: What is the behavior of $f'(u^\ast)$ near the singularities? Does it look like $C/r^2$?

This question is motivated by the fact that in the explicit examples $\Omega=B_1$ and $f(u)=(1+u)^p$, $p>1$ or $f(u)=e^u$ it is always $f'(u^\ast (r))=C/r^2$ for certain positive constant $C$, when the extremal solution $u^\ast$ is singular.

In this paper we give a negative answer to this question, by showing that, in the case in which $\Omega=B_1$ and $u^\ast$ is singular, we always have $\limsup_{r\to 0}r^2f'(u^\ast(r))\in (0,+\infty)$. However, it is possible to give examples of $f\in C^\infty ([0,+\infty ))$ satisfying (\ref{convexa}) for which $u^\ast$ is singular and $\liminf_{r\to 0}r^2f'(u^\ast(r))=0$. In fact, we exhibit a large family of functions $f\in C^\infty ([0,+\infty ))$ satisfying (\ref{convexa}) for which $u^\ast$ is singular and $f'(u^\ast)$ can have a very oscillating behavior.

\begin{theorem}\label{limsup}

Assume that $\Omega=B_1$, $N\geq 10$, and that $f$ satisfies (\ref{convexa}). Suppose that the extremal solution $u^\ast$ of $(P_\lambda)$ is unbounded. 

Then $\limsup_{r\to 0} r^2 f'(u^\ast (r))\in (0,\infty)$. Moreover

$$\frac{2(N-2)}{\lambda^\ast}\leq \limsup_{r\to 0} r^2 f'(u^\ast (r))\leq \frac{\lambda_1}{\lambda^\ast} ,$$

\noindent where $\lambda_1$ denotes the first eigenvalue of the linear problem $-\Delta v=\lambda v$ in $B_1\subset {\mathbb R}^N$ with Dirichlet conditions $v=0$ on $\partial B_1$.

\end{theorem}

\begin{theorem}\label{liminf}

Assume that $\Omega=B_1$, $N\geq 10$, and that $\varphi :(0,1)\rightarrow {\mathbb R^+}$ satisfies $\lim_{r\to 0} \varphi (r)=+\infty$. Then there exists $f\in C^\infty([0,+\infty))$ satisfying (\ref{convexa}) such that the extremal solution $u^\ast$ of $(P_\lambda)$ is unbounded and 

$$\liminf_{r\to 0} \frac{f'(u^\ast (r))}{\varphi (r)}=0.$$

\end{theorem}

Note that in the case $\varphi(r)=1/r^2$, we would obtain $\liminf_{r\to 0} r^2 f'(u^\ast (r))=0$. This answers negatively to \cite[Open problem 5]{BV}. In fact $r^2 f'(u^\ast (r))$ could be very oscillating, as the next result shows.

\begin{theorem}\label{oscillation}
Assume that $\Omega=B_1$, $N\geq 10$, and let $0\leq C_1\leq C_2$, where $C_2\in[2(N-2),(N-2)^2/4]$. Then there exists $f\in C^\infty([0,+\infty))$ satisfying (\ref{convexa}) such that the extremal solution $u^\ast$ of $(P_\lambda)$ is unbounded, $\lambda^\ast=1$ and 

$$\liminf_{r\to 0} r^2 f'(u^\ast (r))=C_1, $$

$$\limsup_{r\to 0} r^2 f'(u^\ast (r))=C_2. $$

\end{theorem}

Note that if $C_1=C_2$, then the interval $[2(N-2),(N-2)^2/4]$ is optimal: $C_2\geq 2(N-2)$ by Theorem \ref{limsup}, while $C_1\leq (N-2)^2/4$ by Hardy's inequality.

\begin{theorem}\label{cualquiera}

Assume that $\Omega=B_1$, $N\geq 11$, and that $\Psi\in C(\overline{B_1}\setminus \{ 0\} )$ is a radially symmetric decreasing function satisfying
$$\frac{2(N-2)}{r^2}\leq \Psi(r) \leq \frac{(N-2)^2}{4 r^2}, \ \ \mbox{ for every } 0<r\leq 1.$$

Then there exist $f\in C^1([0,+\infty))$ satisfying (\ref{convexa}) such that $\lambda^\ast =1$ and 

$$f'(u^\ast (x))=\Psi (x), \ \ \mbox{ for every } x\in \overline{B_1}\setminus\{ 0\}.$$

Moreover, this function $f$ is unique up to a multiplicative constant. That is, if $g$ is a function with the above properties, then there exists $\alpha>0$ such that $g=\alpha \, f(\cdot /\alpha)$ (whose extremal solution is $\alpha u^\ast$).

\end{theorem}

\section{Proof of the main results}

First of all, if $\Omega=B_1$, and $f$ satisfies (\ref{convexa}), it is easily seen by the Gidas-Ni-Nirenberg symmetry result that $u_\lambda$, the solution of $(P_\lambda)$, is radially decreasing for $0<\lambda<\lambda^\ast$. Hence, its limit $u^\ast$ is also radially decreasing. In fact $u_r^\ast(r)<0$ for all $r\in (0,1]$, where $u_r$ denotes the radial derivative of a radial function $u$. Moreover, it is immediate that the minimality of $u_\lambda$ implies its stability. Clearly, we can pass to the limit and obtain that $u^\ast$ is also stable, which means

\begin{equation}\label{inequa}
\int_{B_1}  \vert \nabla \xi\vert^2 \, dx\geq\int_{B_1} \lambda^\ast f'(u^\ast)\xi^2 \, dx
\end{equation}
\noindent for every $\xi\in C^\infty (B_1)$ with compact support in $B_1$.

On the other hand, differentiating $-\Delta u^\ast =\lambda^\ast f(u^\ast)$ with respect to $r$, we have
\begin{equation}\label{ahiledao}
-\Delta u_r^\ast=\left(\lambda^\ast f'(u^\ast) -\frac{N-1}{r^2}\right) u_r^\ast, \ \ \mbox{ for all }r\in (0,1].
\end{equation}

\begin{proposition}\label{key}

Let $N\geq 3$ and $\Psi:\overline{B_1}\setminus\{ 0\} \rightarrow {\mathbb R}$ be a radially symmetric function satisfying that there exists $C>0$  such that $\vert \Psi (r)\vert /r^2 \leq C$,  for every $0<r\leq 1$, and

\begin{equation}\label{ineq}
\int_{B_1}  \vert \nabla \xi\vert^2 \, dx\geq\int_{B_1} \Psi\, \xi^2 \, dx
\end{equation}
\noindent for every $\xi\in C^\infty (B_1)$ with compact support in $B_1$.

Then

\begin{enumerate}

\item[i)] The problem

$$
\left\{
\begin{array}{rll}
-\Delta \omega(x)&={\displaystyle \left( \Psi (x)-\frac{N-1}{\vert x\vert^2}\right) \omega (x)} \ \ \ \ \ \  & \mbox{ in } B_1 \, ,\\
\omega (x)&= 1 & \mbox{ on } \partial B_1 \, ,\\
\end{array}
\right. \eqno{(P_\Psi})
$$

\noindent has an unique solution $\omega\in W^{1,2}(B_1)$. Moreover $\omega$ is radial and strictly positive in $B_1\setminus \{ 0\}$ .

 \
 
 \item[ii)] If  $\Psi_1 \leq \Psi_2$ in $\overline{B_1}\setminus \{ 0\} $ satisfy the above hypotheses and $\omega_i$ $(i=1,2)$ are the solutions of the problems 
$(P_{\Psi_i})$  then $\omega_1 \leq \omega_2$ in $\overline{B_1}\setminus \{ 0\}$.

\end{enumerate}

\end{proposition}

\begin{proof} 
 i) By Hardy's inequality
 
$$\int_{B_1}  \vert \nabla \xi\vert^2 \, dx\geq \frac{(N-2)^2}{4}\int_{B_1}\frac{\xi^2}{\vert x\vert^2} \, dx,$$

\noindent for every $\xi\in C^\infty (B_1)$ with compact support in $B_1$,  we can define the functional $I:X\rightarrow {\mathbb R}$ by
 
 $$I(\omega):=\frac{1}{2}\int_{B_1} \vert \nabla \omega \vert^2 dx-\frac{1}{2}\int_{B_1} \left( \Psi-\frac{N-1}{\vert x\vert^2}\right) \omega^2 dx,$$
 
 \noindent for every $\omega\in X$, where $X=\left\{ \omega:B_1\rightarrow {\mathbb R} \mbox{ such that } \omega-1\in W_0^{1,2}(B_1)\right\} $.
 
 It is immediate that 
 
 $$I'(\omega)(v)=\int_{B_1}\nabla \omega \nabla v \, dx-\int_{B_1} \left( \Psi-\frac{N-1}{\vert x\vert^2}\right) \omega v\, dx\, ; \ \ \ \omega\in X,v\in W_0^{1,2}(B_1).$$
 
 Therefore to prove the existence of a solution of (P$_\Psi$) it is sufficient to show that $I$ has a global minimum in $X$. To do this, we first prove that $I$ is bounded from below in $X$. Taking $v=\omega -1$ in (\ref{ineq}) and applying Cauchy–Schwarz inequality we obtain
 
 $$I(\omega)\geq \frac{1}{2}\int_{B_1}\Psi (\omega-1)^2 dx-\frac{1}{2}\int_{B_1} \left( \Psi-\frac{N-1}{\vert x\vert^2}\right) \omega^2 dx=$$
 
 $$=\frac{1}{2}\int_{B_1} \Psi (-2\omega +1) dx+\frac{1}{2}\int_{B_1} \frac{N-1}{\vert x\vert^2} \omega^2 dx$$
 
 $$\geq\frac{1}{2}\int_{B_1} \frac{-C (2\vert\omega\vert+1)+(N-1)\omega^2}{\vert x\vert^2}dx\geq\frac{1}{2}\int_{B_1}\frac{-C-C^2}{\vert x\vert^2}\, dx.$$
 
 Hence $I$ is bounded from below in $X$. Take $\{ w_n\}\subset X$ such that $ I(\omega_n)\rightarrow \inf I$. Let us show that  $\{ w_n\}$ is bounded in $W^{1,2}$. To this end, taking into account the above inequalities and that $-C(2\vert s\vert+1)+(N-1)s^2\geq -C(2\vert s\vert+1)+2s^2\geq s^2-C-C^2$ for every $N\geq 3$ and $s\in  {\mathbb R}$, we have
 
 $$I(\omega_n)\geq \frac{1}{2}\int_{B_1} \frac{-C(2 \vert\omega_n\vert+1)+(N-1)\omega_n^2}{\vert x\vert^2}dx\geq\frac{1}{2}\int_{B_1}\frac{\omega_n^2-C-C^2}{\vert x\vert^2}\, dx.$$
 
 From this $\int_{B_1}\omega_n^2/\vert x\vert^2$ is bounded. Therefore  $\int_{B_1}\Psi\omega_n^2$ is also bounded. From the definition of $I$ we conclude that $\int_{B_1}\vert \nabla \omega_n\vert^2$ is bounded, which clearly implies that $\{ w_n\}$ is bounded in $W^{1,2}$. 
 
 Since $X$ is a weakly closed subset of $W^{1,2}$, we have that, up to a subsequence, $\omega_n \rightharpoonup \omega_0\in X$. Taking $v=\omega_n-\omega_0$ in (\ref{ineq}) we deduce
 
 \
 
 $I(\omega_n)-I(\omega_0)$
 $$=\frac{1}{2}\int_{B_1} \vert \nabla (\omega_n-\omega_0) \vert^2 dx-\frac{1}{2}\int_{B_1} \Psi (\omega_n-\omega_0)^2 dx+\frac{1}{2}\int_{B_1} \frac{(N-1)(\omega_n-\omega_0)^2}{\vert x\vert^2} dx$$
 $$+\int_{B_1}\nabla \omega_0\nabla (\omega_n-\omega_0) dx-\int_{B_1}\Psi \omega_0 (\omega_n-\omega_0) dx+\int_{B_1}\frac{(N-1)\omega_0 (\omega_n-\omega_0)}{\vert x\vert^2}dx$$
 $$\geq \int_{B_1}\nabla \omega_0\nabla (\omega_n-\omega_0) dx-\int_{B_1}\Psi \omega_0 (\omega_n-\omega_0) dx+\int_{B_1}\frac{(N-1)\omega_0 (\omega_n-\omega_0)}{\vert x\vert^2}dx.$$
 
Since $\omega_n-\omega_0 \rightharpoonup 0$, taking limit as $n$ tends to infinity in the above inequality we conclude

$$(\inf I)-I(\omega_0)\geq 0,$$

\noindent which implies that $I$ which attains its minimum at $\omega_0$. The existence of solution of  (P$_\Psi$) is proven.

To show the uniqueness of solution suppose that there exists two solutions $\omega_1$ and $\omega_2$ of the same problem  (P$_\Psi$). Then $\omega_2-\omega_1\in W_0^{1,2}$. By (\ref{ineq}) we have
 
 $$0=I'(\omega_2)(\omega_2-\omega_1)-I'(\omega_1)(\omega_2-\omega_1)$$
 $$=\int_{B_1} \vert \nabla (\omega_2-\omega_1) \vert^2 dx-\int_{B_1} \Psi (\omega_2-\omega_1)^2 dx+\int_{B_1} \frac{(N-1)(\omega_2-\omega_1)^2}{\vert x\vert^2} dx$$
 $$\geq \int_{B_1} \frac{(N-1)(\omega_2-\omega_1)^2}{\vert x\vert^2} dx,$$
 
\noindent which implies that $\omega_1=\omega_2$. The uniqueness is proven.
 
The radial symmetry of the solution of (P$_\Psi$)  follows easily from the uniqueness of solution and the radiality of the function $\Psi(x)-(N-1)/\vert x\vert^2$ and the boundary condition of the problem.

Finally, to prove that the solution $\omega$ of (P$_\Psi$) is strictly positive in $B_1\setminus \{ 0\}$ suppose, contrary to our claim, that there exists $r_0\in (0,1)$ such that $\omega(r_0)=0$ (with radial notation). Thus the function $v$ defined by $v=\omega$ in $B_{r_0}$ and $v=0$ in $B_1\setminus \overline{B_{r_0}}$ is in $W_0^{1,2}(B_1)$. By (\ref{ineq}) we have 
 
 $$0=I'(\omega)(v)=\int_{B_{r_0}}\vert \nabla \omega\vert^2 dx-\int_{B_{r_0}}\Psi \omega^2 dx+\int_{B_{r_0}}\frac{(N-1)\omega^2}{\vert x\vert^2}dx$$
 $$\geq \int_{B_{r_0}}\frac{(N-1)\omega^2}{\vert x\vert^2}dx.$$
 
 Therefore $\omega=0$ in $B_{r_0}$. In particular $\omega(r_0)=\omega'(r_0)=0$ (with radial notation), which implies, by the uniqueness of the corresponding Cauchy problem, that $\omega=0$ in $(0,1]$. This contradicts $\omega(1)=1$.
 
 \
 
 ii) Consider the function $v=(\omega_1-\omega_2)^+=\max\{0,\omega_1-\omega_2\}\in W_0^{1,2}(B_1)$ in the weak formulation of problem (P$_{\Psi_1}$). We have
 
 $$0=\int_{B_1}\left(\nabla \omega_1 \nabla (\omega_1-\omega_2)^+ -\Psi_1 \omega_1 (\omega_1-\omega_2)^+ +\frac{(N-1)\omega_1 (\omega_1-\omega_2)^+ }{\vert x\vert^2}\right) dx$$
 
Consider the same function $v=(\omega_1-\omega_2)^+$ in the weak formulation of problem (P$_{\Psi_2}$). Taking into account that $\Psi_1\leq \Psi_2$ and $\omega_2\geq 0$ we obtain

$$ 0=\int_{B_1}\left(\nabla \omega_2 \nabla (\omega_1-\omega_2)^+ -\Psi_2 \omega_2 (\omega_1-\omega_2)^+ +\frac{(N-1)\omega_2 (\omega_1-\omega_2)^+ }{\vert x\vert^2}\right) dx$$
$$\leq \int_{B_1}\left(\nabla \omega_2 \nabla (\omega_1-\omega_2)^+ -\Psi_1 \omega_2 (\omega_1-\omega_2)^+ +\frac{(N-1)\omega_2 (\omega_1-\omega_2)^+ }{\vert x\vert^2}\right) dx$$

Subtracting the above two expressions it is follows that

$$0\geq\int_{B_1} \vert \nabla (\omega_1-\omega_2)^+ \vert^2 dx-\int_{B_1} \Psi_1 (\omega_1-\omega_2)^{+\, 2} dx+\int_{B_1} \frac{(N-1)(\omega_1-\omega_2)^{+\, 2}}{\vert x\vert^2} dx$$
 $$\geq \int_{B_1} \frac{(N-1)(\omega_1-\omega_2)^{+\, 2}}{\vert x\vert^{2}} dx.$$
 
 This implies $(\omega_1-\omega_2)^+=0$. Hence $\omega_1\leq \omega_2$, which is our claim.
\end{proof}

\noindent\textbf{Proof of Theorem \ref{limsup}.}
We first prove that $\lambda^\ast f'(u^\ast(r))\leq \lambda_1/r^2$ for every $r\in (0,1]$. To see this, let $0<\varphi_1$ be the first eigenfunction of the linear problem $-\Delta v=\lambda v$ in $B_1\subset {\mathbb R}^N$ with Dirichlet conditions $v=0$ on $\partial B_1$. Then $\int_{B_1} \vert\nabla \varphi_1 \vert^2=\lambda_1 \int_{B_1}\varphi_1^2$. By density, for arbitrary $0<r\leq 1$, we could take in (\ref{inequa}) the radial function $\xi=\varphi_1 (\cdot/r)$ in $B_r$ and $\xi=0$ in $B_1 \setminus \overline{B_r}$. Since $f'$ is nondecreasing and $u^\ast$ is radially decreasing, then $f'(u^\ast)$ is radially decreasing. An easy computation shows that

$$ \int_{B_1} \vert\nabla \xi \vert^2=\int_{B_r} \vert\nabla \xi \vert^2=r^{N-2} \int_{B_1} \vert\nabla \varphi_1 \vert^2=\lambda_1 r^{N-2} \int_{B_1}\varphi_1^2\ ,$$

$$\int_{B_1}\lambda^\ast f'(u^\ast) \xi^2=\int_{B_r}\lambda^\ast f'(u^\ast) \xi^2\geq\lambda^\ast f'(u^\ast(r)) \int_{B_r} \xi^2=\lambda^\ast f'(u^\ast(r)) r^N \int_{B_1}\varphi_1^2\ .$$

Combining this with (\ref{inequa}) we obtain the desired conclusion. Consequently $\limsup_{r\to 0} r^2 f'(u^\ast (r))\leq \lambda_1/\lambda^\ast$.

We now prove that $\limsup_{r\to 0} r^2 f'(u^\ast (r))\geq 2(N-2)/\lambda^\ast$. To obtain a contradiction, suppose that  there exists $r_0\in (0,1]$ and $\varepsilon>0$ such that

\begin{equation}\label{ves}
\lambda^\ast f'(u^\ast (r))\leq \frac{2(N-2)-\varepsilon}{r^2}, 
\end{equation}

\noindent for every $r\in (0,r_0]$. Consider now the radial function $\omega (r):=u_r^\ast (r_0\, r)/u_r^\ast (r_0)$, defined in $\overline{B_1}\setminus\{ 0\}$. Applying (\ref{ahiledao}), an easy computation shows that $\omega(1)=1$ and

$$-\Delta \omega (r)=\frac{1}{u_r^\ast (r_0)}r_0^2\left( -\Delta (u_r^\ast (r_0\, r))\right)$$
$$=\frac{1}{u_r^\ast (r_0)}r_0^2\left( \lambda^\ast f'(u^\ast (r_0\, r))-\frac{N-1}{(r_0\, r)^2}\right) u_r^\ast (r_0\, r)=\left( \Psi (r)-\frac{N-1}{r^2}\right)\omega(r),$$

\

\noindent for every $r\in (0,1)$, where $\Psi(r):=r_0^2\lambda^\ast f'(u^\ast (r_0\, r))$. From (\ref{ves}) we obtain $\Psi(r)\leq \Psi_2(r):=(2(N-2)-\varepsilon)/r^2$ for every $r\in (0,1]$. It is easy to check that the solution $\omega_2$ of the problem $(P_{\Psi_2})$ is given by $w_2(r)=r^\alpha$ ($0<r\leq 1$) where
$$\alpha=\frac{2-N+\sqrt{(N-4)^2+4\varepsilon}}{2}.$$
Therefore, applying Proposition \ref{key}, we can assert that $0<\omega (r)\leq r^\alpha$ for every $r\in (0,1]$. It is clear that $\alpha>-1$. Hence $\omega\in L^1(0,1)$. This gives $u_r^\ast \in L^1(0,r_0)$, which contradicts the unboundedness of $u^\ast$. \qed

\begin{lemma}\label{AB}

Let $N\geq 10$ and $0<A<B\leq 1$. Define the radial function $\Psi_{A,B}:\overline{B_1}\setminus\{ 0\} \rightarrow {\mathbb R}$ by 

$$\Psi_{A,B}(r):=\left\{
\begin{array}{ll}
0 & \mbox{ if } 0< r <A \, \\ \\
\displaystyle{\frac{2(N-2)}{r^2}} & \mbox{ if } A\leq r\leq B \, ,\\ \\
0 & \mbox{ if } B<r\leq 1.
\end{array}
\right.
$$

Let $\omega[A,B]$ be the unique radial solution of $(P_{\Psi_{A,B}})$. Then

$$\lim_{s \to 0}\int_0^1\omega[s e^{-1/s^3},s](r)dr=+\infty.$$

\end{lemma}

\begin{proof}
We first observe that since $N\geq 10$ we have $2(N-2)\leq (N-2)^2/4$. Hence $0\leq \Psi_{A,B}\leq (N-2)^2/(4r^2)$ for every $0<r\leq 1$.  Thus, by Hardy's inequality, $\Psi_{A,B}$ satisfies (\ref{ineq}) and we can apply Proposition \ref{key}.

We check at once that

$$\omega[A,B](r)=\left\{
\begin{array}{ll}\frac{N(N-4)B^{N-2}A^{-2}\ r}{(N-2)^2B^{N-4}-4A^{N-4}+2(N-2)B^N(B^{N-4}-A^{N-4})} & \mbox{ if } 0\leq r <A ,\\ \\
\frac{N(N-2)B^{N-2}\ r^{-1}\ -\ 2NA^{N-4}B^{N-2}\ r^{3-N}}{(N-2)^2B^{N-4}-4A^{N-4}+2(N-2)B^N(B^{N-4}-A^{N-4})}  & \mbox{ if } A\leq r\leq B , \\ \\
\frac{\left( (N-2)^2B^{N-4}-4A^{N-4}\right) \ r\ +\ 2(N-2)B^N(B^{N-4}-A^{N-4})\ r^{1-N}}{(N-2)^2B^{N-4}-4A^{N-4}+2(N-2)B^N(B^{N-4}-A^{N-4})}  & \mbox{ if } B<r\leq 1.
\end{array}
\right.
$$

\

To see that $\omega[A,B]$ is the solution of (P$_{\Psi_{A,B}}$) it suffices to observe that $\omega[A,B]\in C^1(\overline{B_1}\setminus\{ 0\})\cap W^{1,2}(B_1)$ satisfies pointwise  (P$_{\Psi_{A,B}}$) if $\vert x\vert \neq A,B$.

On the other hand, taking into account that $r^{3-N}\leq A^{4-N}r^{-1}$ if $A\leq r\leq B$, we have that

$$\omega[A,B](r)\geq \frac{N(N-2)B^{N-2}\ r^{-1}\ -\ 2NA^{N-4}B^{N-2}A^{4-N}\ r^{-1}}{(N-2)^2B^{N-4}-4A^{N-4}+2(N-2)B^N(B^{N-4}-A^{N-4})}$$

$$\geq \frac{N(N-2)B^{N-2}\ r^{-1}\ -\ 2NA^{N-4}B^{N-2}A^{4-N}\ r^{-1}}{(N-2)^2B^{N-4}+2(N-2)B^N B^{N-4}}$$

$$=\frac{N(N-4)B^2 \ r^{-1}}{(N-2)^2+2(N-2)B^N} \, ,\ \ \mbox{  if } A\leq r\leq B.$$

\

From this and the positiveness of $\omega[A,B]$ it follows that

$$\int_0^1\omega[A,B](r)\geq\int_A^B\omega[A,B](r)dr\geq \int_A^B \frac{N(N-4)B^2 \ r^{-1}}{(N-2)^2+2(N-2)B^N}dr$$

$$=\frac{N(N-4)B^2 \  \log (B/A)}{(N-2)^2+2(N-2)B^N}.$$

\

Taking in this inequality $A=s e^{-1/s^3}$, $B=s$ (for arbitrary $0<s\leq 1$),
 it may be concluded that
 
 $$\int_0^1\omega[s e^{-1/s^3},s](r)dr\geq\frac{N(N-4)}{s\left( (N-2)^2+2(N-2)s^N\right)}$$
 
 \
 
\noindent and the lemma follows.
 \end{proof}
 
 \begin{proposition}\label{peasofuncion}
 
Let $N\geq 10$ and $\varphi :(0,1)\rightarrow {\mathbb R}^+$ such that $\lim_{r\to 0} \varphi (r)=+\infty$. Then there exists $\Psi\in C^\infty (\overline{B_1}\setminus \{ 0\})$ an unbounded radially symmetric decreasing function satisfying
 
 \begin{enumerate}
 \item[i)] $\displaystyle{0<\Psi(r)\leq\ \frac{2(N-2)}{r^2}}$ and $\Psi'(r)<0$ for every $0<r\leq 1$.
 
 \item[ii)] $\displaystyle{\liminf_{r\to 0} \frac{\Psi(r)}{\varphi (r)}=0}$, $\displaystyle{\limsup_{r\to 0} r^2 \Psi (r)=2(N-2)}$.
 
 \item[iii)] $\displaystyle{\int_0^1 \omega(r)dr=+\infty}$, where $\omega$ is the radial solution of (P$_\Psi$).

 \end{enumerate}

 \end{proposition}
 
 \begin{proof}
 Without loss of generality we can assume that $\varphi (r)\leq 2(N-2)/r^2$ for $r\in (0,1]$, since otherwise we can replace $\varphi$ with $\overline{\varphi}=\min\left\{ \varphi, 2(N-2)/r^2\right\} $. It is immediate that  $\lim_{r\to 0} \varphi (r)=+\infty$ implies $\lim_{r\to 0} \overline{\varphi}(r)=+\infty$ and that  $0\leq \liminf_{r\to 0}\Psi(r)/\varphi (r)\leq\liminf_{r\to 0} \Psi(r)/\overline{\varphi } (r)$.
 
 We begin by constructing by induction two sequence $\{x_n\}$, $\{y_n\}\subset (0,1]$ in the following way: $x_1=1$ and, knowing the value of $x_n$ $(n\geq 1)$, take $y_n$ and $x_{n+1}$ such that 
 
 $$x_{n+1}<y_n<x_n e^{-1/x_n^3}<x_n,$$
 
 \
 
 \noindent where $y_n\in (0, x_n e^{-1/x_n^3})$ is chosen such that
 
 $$\varphi(y_n)>(n+1)\frac{2(N-2)}{\left(x_n e^{-1/x_n^3}\right)^2},$$
 
 \noindent which is also possible since $\lim_{r\to 0} \varphi (r)=+\infty$. The inequality $x_{n+1}<x_n e^{-1/x_n^3}$ for every integer $n\geq 1$ implies that $\{ x_n \}$ is a decreasing sequence tending to zero as $n$ goes to infinity. For this reason, to construct the radial function $\Psi$ in $B_1\setminus \{ 0\}$, it suffices to define $\Psi$ in every interval  $[x_{n+1},x_n)=[x_{n+1},y_n)\cup [y_n, x_n e^{-1/x_n^3}]\cup (x_n e^{-1/x_n^3},x_n)$. 
 
 First, we define
 
 $$\Psi(r):=\frac{2(N-2)}{r^2}, \ \ \ \mbox{  if } \  \ x_n e^{-1/x_n^3}<r<x_n,$$
 
 $$\Psi (y_n):=\frac{\varphi (y_n)}{n+1}.$$
 
 By the definition of $y_n$ we have that
 
 $$\Psi (y_n)=\frac{\varphi (y_n)}{n+1}>\frac{2(N-2)}{\left(x_n e^{-1/x_n^3}\right)^2}\ \mbox{ and }\ \Psi (y_n)<\varphi(y_n)\leq\frac{2(N-2)}{y_n^2}.$$
 
 Thus, it is a simple matter to see that it is possible to take a decreasing function $\Psi$ in $(y_n, x_n e^{-1/x_n^3}]$ such that $\Psi(r)<2(N-2)/r^2$ and $\Psi'(r)<0$ for $r\in(y_n, x_n e^{-1/x_n^3}]$ and $\Psi \in C^\infty ([y_n,x_n))$.
 
 Finally, we will define similarly $\Psi$ in $[x_{n+1},y_n)$. Taking into account that
 $$\Psi (y_n)<\varphi(y_n)\leq\frac{2(N-2)}{y_n^2}<\frac{2(N-2)}{x_{n+1}^2},$$
 
 \noindent we see at once that  it is possible to take a decreasing function $\Psi$ in $[x_{n+1}, y_n)$ such that 
 
 $$\Psi (x_{n+1})=\frac{2(N-2)}{x_{n+1}^2},$$

 $$\partial_r^{(k)} \Psi (x_{n+1})=\partial_r^{(k)} \left(2(N-2)/r^2\right)(x_{n+1}), \ \ \mbox{for every } k\geq 1,$$
 
 $$\Psi(r)<2(N-2)/r^2 \ \mbox{ and } \ \Psi'(r)<0 \ \ \ \mbox{for }r\in(x_{n+1},y_n),$$
 
 $$\Psi \in C^\infty ([x_{n+1},x_n)).$$
 
 Once we have constructed the radial function $\Psi$ it is evident that  $\Psi\in C^\infty (\overline{B_1}\setminus \{ 0\})$ an unbounded radially symmetric decreasing function satisfying i).
 
 To prove ii) it is sufficient to observe that the sequences  $\{x_n \}$, $\{ y_n \}$ tend to zero and satisfy $x_n^2 \Psi(x_n)=2(N-2)$ and $\Psi (y_n)/\varphi(y_n)=1/(n+1)$ for every integer $n\geq 1$.
 
It remains to prove iii). To this end consider an arbitrary $K>0$. Since $\{x_n \}$ tends to zero, applying Lemma \ref{AB} we can assert that there exists a natural number $m$ such that

$$\int_0^1\omega[x_m e^{-1/{x_m}^3},x_m](r)dr\geq K.$$

\

Observe that $\Psi\geq \Psi_{x_m e^{-1/{x_m}^3},x_m}$. By Proposition \ref{key} it follows that $\omega\geq\omega[x_m e^{-1/{x_m}^3},x_m]$. Thus

$$\int_0^1\omega (r) dr\geq \int_0^1\omega[x_m e^{-1/{x_m}^3},x_m](r)dr\geq K.$$

Since $K>0$ is arbitrary we conclude $\int_0^1\omega (r) dr=+\infty$.
\end{proof}

\

\noindent\textbf{Proof of Theorem \ref{liminf}.}
Consider the function $\Psi$ of Proposition \ref{peasofuncion} and let $\omega$ be the radial solution of $(P_\Psi)$. Since  $\Psi\in C^\infty (\overline{B_1}\setminus \{ 0\})$ we obtain $\omega\in C^\infty (\overline{B_1}\setminus \{ 0\})\cap W^{1,2}(B_1)$. Define the radial function $u$ by

$$u(r):=\int_r^1 \omega (t)dt, \  \ 0<r\leq 1.$$

It is obvious that $u\in C^\infty (\overline{B_1}\setminus \{ 0\})$. Since $u'=-\omega$ (with radial notation), we have $u\in W^{2,2}(B_1)\subset W^{1,2}(B_1)$. Moreover, from $\int_0^1 \omega(r)dr=+\infty$ we see that $u$ is unbounded.

On the other hand, since $u'=-\omega<0$ in $(0,1]$ (by Proposition \ref{key}), it follows that $u$ is a decreasing $C^\infty$ diffeomorphism between $(0,1]$ and $[0,+\infty)$. Therefore we can define $f\in C^\infty ([0,+\infty))$ by

$$f:=(-\Delta u)\circ u^{-1}.$$

\

We conclude that $u\in  W_0^{1,2} (B_1)$ is an unbounded solution of (P$_\lambda$) for $\lambda=1$.

Now, substituting $u_r$ by $-\omega$ in (\ref{ahiledao}) it follows that

$$-\Delta (-\omega)+f'(u)(-\omega)=\frac{N-1}{r^2}(-\omega) \ \ \mbox{ for } 0<r\leq 1$$.

Hence, since $\omega$ is a solution of (P$_\Psi$) we obtain $f'(u)\omega=\Psi \omega$ in $(0,1]$. From $\omega>0$ in $(0,1]$ we conclude that

$$f'(u(x))=\Psi(x)\ \ \ \mbox{ for every } x\in \overline{B_1}\setminus \{ 0\}.$$

We now prove that $f$ satisfies (\ref{convexa}). To do this, we first claim that $\omega'(1)\geq -1$. Since $\Psi\leq 2(N-2)/r^2$, applying Proposition \ref{key} with $\Psi_1=\Psi$ and $\Psi_2=2(N-2)/r^2$, we deduce $\omega_1\leq \omega_2$, where $\omega_1=\omega$ and $\omega_2=r^{-1}$, as is easy to check. Since $\omega_1(1)=\omega_2(1)$ it follows $\omega_1'(1)\geq\omega_2'(1)=-1$, as claimed.

Thus 

$$f(0)=f(u(1))=-\Delta u(1)=-u''(1)-(N-1)u'(1)=\omega'(1)+(N-1)\omega(1)$$
$$\geq (-1)+(N-1)>0.$$

On the other hand, since $f'(u(r))=\Psi(r)>0$ for every $r\in (0,1]$ it follows $f'>0$ in $[0,+\infty)$. Moreover $\lim_{s\to+\infty}f'(s)=\lim_{r\to 0}f'(u(r))=\lim_{r\to 0}\Psi(r)=+\infty$, and the superlinearity of $f$ is proven. Finally, to show the convexity of $f$,  it suffices to differentiate the expression $f'(u)=\Psi$ with respect to $r$ (with radial notation), obtaining $u'(r)f''(u(r))=\Psi'(r)$ in $(0,1]$. Since $u'<0$ and $\Psi'<0$ we obtain $f''(u(r))>0$ in $(0,1]$, which gives the convexity of $f$ in $[0,+\infty)$.

Finally, we show that $u$ is a stable solution of $(P_\lambda)$ for $\lambda=1$. Since $N\geq 10$ then $2(N-2)\leq (N-2)^2/4$, hence

$$f'(u(r))=\Psi(r)\leq \frac{2(N-2)}{r^2}\leq \frac{(N-2)^2}{4r^2}\ \ \mbox{ for every } 0<r\leq 1.$$

Thus, by Hardy's inequality, we conclude that $u$ is a stable solution of $(P_\lambda)$ for $\lambda=1$. 

On the other hand, in \cite[Th. 3.1]{BV} it is proved that if $f$ satisfies (\ref{convexa}) and $u\in W_0^{1,2}(\Omega)$ is an unbounded stable weak solution of ($P_\lambda$) for some $\lambda>0$, then $u=u^\ast$ and $\lambda=\lambda^\ast$. Therefore we conclude that $\lambda^\ast=1$, $u^\ast=u$ and

$$\liminf_{r\to 0}\frac{f'(u^\ast(r))}{\varphi (r)}=\liminf_{r\to 0}\frac{\Psi(r)}{\varphi(r)}=0.$$ \qed

\noindent\textbf{Proof of Theorem \ref{oscillation}.}
Take $\varphi(r)=1/r^2$, $0<r\leq1$, and consider the function $\Psi$ of Proposition \ref{peasofuncion}. Define

$$\Phi(r):=\frac{C_2-C_1}{2(N-2)}\Psi(r)+\frac{C_1}{r^2},$$

\noindent for every $0<r\leq 1$. Then it follows easily that $\Phi\in C^\infty (\overline{B_1}\setminus \{ 0\})$ is an unbounded radially symmetric decreasing function satisfying
 
 \begin{enumerate}
 \item[i)] $\displaystyle{\Psi(r)\leq\Phi(r)\leq \frac{(N-2)^2}{4r^2}}$ and $\Phi'(r)<0$ for every $0<r\leq 1$.
 
 \item[ii)] $\displaystyle{\liminf_{r\to 0} r^2 \Phi(r)=C_1}$, $\displaystyle{\limsup_{r\to 0} r^2 \Phi (r)=C_2}$.
 
 \item[iii)] $\displaystyle{\int_0^1 \varpi(r)dr=+\infty}$, where $\varpi$ is the radial solution of (P$_\Phi$).
  \end{enumerate}
  
  Note that iii) follows from Proposition \ref{key}, Proposition \ref{peasofuncion}  and the fact that $\varpi\geq\omega$, being $\omega$ the  radial solution of $(P_\Psi$).
 
 The rest of the proof is very similar to that of Theorem \ref{liminf}. Since $\Phi\in C^\infty(\overline{B_1}\setminus \{ 0\})$ we obtain $\varpi\in C^\infty (\overline{B_1}\setminus \{ 0\})\cap W^{1,2}(B_1)$. Define the radial function $u$ by

$$u(r):=\int_r^1 \varpi (t)dt, \  \ 0<r\leq 1.$$

Analysis similar to that in the proof of Theorem \ref{liminf}  shows that $u\in W^{2,2}$ is a decreasing $C^\infty$ diffeomorphism between $(0,1]$ and $[0,+\infty)$. Defining again $f:=(-\Delta u)\circ u^{-1}$, it is obtained that $f\in C^\infty ([0,+\infty))$. Thus  $u\in  W_0^{1,2} (B_1)$ is an unbounded solution of $(P_\lambda )$ for $\lambda=1$. It remains to prove that $f$ satisfies (\ref{convexa}). At this point, the only difference with respect to the proof of Theorem \ref{liminf} is that  $\Phi(r)\leq\Psi_2(r):=(N-2)^2/(4r^2)$ implies that $\varpi\leq\omega_2$, being $\omega_2(r)=r^{-N/2+\sqrt{N-1}+1}$ the solution of the problem $(P_{\Psi_2})$. Hence $\varpi'(1)\geq\omega_2'(1)=-N/2+\sqrt{N-1}+1$. Therefore

$$f(0)=f(u(1))=-\Delta u(1)=-u''(1)-(N-1)u'(1)=\varpi'(1)+(N-1)\varpi(1)$$
$$\geq (-N/2+\sqrt{N-1}+1)+(N-1)>0.$$

The rest of the proof runs as before. \qed

\

\noindent\textbf{Proof of Theorem \ref{cualquiera}.}
Since $0<\Psi\leq (N-2)^2/(4r^2)$ we have that $\Psi$ satisfies the hypothesis of Proposition \ref{key}. Thus we can consider the solution $\omega$ of the problem $(P_\Psi)$. From $\Psi\in C(\overline{B_1}\setminus \{ 0\})$ it follow that $\omega\in C^2(\overline{B_1}\setminus \{ 0\})\cap W^{1,2} (B_1)$. On the other hand, since $\Psi(r)\geq \Psi_1(r):=2(N-2)/r^2$ for $0<r\leq 1$, we have that $\omega(r)\geq\omega_1(r):=r^{-1}$ for  $0<r\leq 1$, where have used that $\omega_1$ is the solution of $(P_{\Psi_1})$ and we have applied Proposition \ref{key}. Define the radial function $u$ by

$$u(r):=\int_r^1 \omega (t)dt, \  \ 0<r\leq 1.$$

Therefore $u(r)\geq\vert\log r\vert$ for $0<r\leq 1$. In particular, $u$ is unbounded. From been proved, it follows that $u\in C^3(\overline{B_1}\setminus \{ 0\})\cap W^{2,2} (B_1)$. Hence (with radial notation) we have that $u$ is a decreasing $C^3$ diffeomorphism between $(0,1]$ and $[0,+\infty)$. Thus we can define $f\in C^1 ([0,+\infty))$ by

$$f:=(-\Delta u)\circ u^{-1}.$$

Analysis similar to that in the proof of Theorems \ref{liminf} and \ref{oscillation} shows that $f$ satisfies (\ref{convexa}), $\lambda^\ast=1$ and $u=u^\ast$.

Finally, to prove that $f$ is unique up to a multiplicative constant, suppose that $g$ is a function satisfying (\ref{convexa}), $\lambda^\ast=1$ and $g'(v^\ast(x))=\Psi (x)$, for every $x\in \overline{B_1}\setminus \{ 0\}$, where $v^\ast$ is the extremal solution associated to $g$.  From (\ref{ahiledao}) we see that

$$ -\Delta v_r^\ast=\left(g'(v^\ast) -\frac{N-1}{r^2}\right) v_r^\ast, \ \ \mbox{ for all }r\in (0,1].$$

It follows immediately that $v_r^\ast (r)/v_r^\ast (1)$ is the solution of the problem $(P_\Psi)$. Since this problem has an unique solution we deduce that  $v_r^\ast (r)/v_r^\ast (1)=\omega(r)=-u_r^\ast(r)$, for every $r\in (0,1]$. Thus $v_r^\ast =\alpha u_r\ast$ for some $\alpha >0$, which implies, since $v^\ast(1)=u^\ast(1)=0$, that $v^\ast =\alpha u^\ast$. The proof is completed by showing that

$$g(v^\ast (x))=-\Delta v^\ast (x)=\alpha (-\Delta u^\ast (x))=\alpha f(u^\ast(x))=\alpha f(v^\ast (x)/\alpha),$$

\noindent for every $x\in\overline{B_1}\setminus \{ 0\})$ and taking into account that $v^\ast\left(\overline{B_1}\setminus \{ 0\}\right)=[0,+\infty)$. \qed

\end{document}